\documentclass[10pt,a4paper]{article}
\pdfoutput=1

\usepackage{lmodern}
\usepackage[T1]{fontenc}

\usepackage{amsmath}
\usepackage{amssymb}
\numberwithin{equation}{section}
\usepackage{amsthm}
\usepackage{mathscinet}

\theoremstyle{plain}
\newtheorem{proposition}{Proposition}[section]
\newtheorem{corollary}[proposition]{Corollary}
\newtheorem{lemma}[proposition]{Lemma}
\newtheorem{theorem}[proposition]{Theorem}
\theoremstyle{definition}
\newtheorem{definition}[proposition]{Definition}
\newtheorem{example}[proposition]{Example}
\newtheorem{remark}[proposition]{Remark}

\DeclareMathOperator{\sgn}{sgn}
\DeclareMathOperator{\diag}{diag}
\DeclareMathOperator{\tr}{tr}
\DeclareMathOperator{\card}{card}

\newcommand{\Z}{\mathbf{Z}}
\newcommand{\abs}[1]{\left\lvert #1 \right\rvert}

\newcommand{\smallbullet}{\,\vcenter{\hbox{\tiny$\bullet$}}\,}
\newcommand{\minor}[1]{\left\lvert #1 \right\rvert}
\newcommand{\separation}[1]{\operatorname{sep}(#1)}
\newcommand{\detterm}[2]{#1_{(#2)}}

\usepackage{tikz}
\tikzset{
  source/.style={circle,draw=black!100,fill=black!50,inner sep = 0,minimum size=2mm},
  sink/.style={circle,draw=black!100,fill=white,inner sep = 0,minimum size=2mm}
}

\newcommand{\bipartite}[1]{%
  \foreach \n in {1,...,#1} {
    \node (a\n) at (0,\n) [source,label={180:$\n$}] {};
    \node (b\n) at (3,\n) [sink,label={0:$\n$}] {};
  }}

\usepackage{microtype}

\usepackage[colorlinks,linkcolor=blue,citecolor=blue,pdfauthor={Daniel Gomez, Hans Lundmark, Jacek Szmigielski},pdftitle={The Canada Day Theorem}]{hyperref}

\begin{document}

\title{The Canada Day Theorem}

\author{%
  Daniel Gomez\thanks{Department of Mathematics and Statistics, University of Saskatchewan, 106 Wiggins Road, Saskatoon, Saskatchewan, S7N 5E6, Canada; dag857@mail.usask.ca}
  \and
  Hans Lundmark\thanks{Department of Mathematics, Link{\"o}ping University, SE-581 83 Link{\"o}ping, Sweden; hans.lundmark@liu.se}
  \and
  Jacek Szmigielski\thanks{Department of Mathematics and Statistics, University of Saskatchewan, 106 Wiggins Road, Saskatoon, Saskatchewan, S7N 5E6, Canada; szmigiel@math.usask.ca}
}

\date{June 16, 2012}

\maketitle

\begin{abstract}
  The Canada Day Theorem is an identity involving sums of $k \times k$ minors
  of an arbitrary $n \times n$ symmetric matrix.
  It was discovered as a by-product of the work on so-called peakon solutions of an
  integrable nonlinear partial differential equation proposed by V.~Novikov.
  Here we present another proof of this theorem, which explains the underlying mechanism
  in terms of the orbits of a certain abelian group action
  on the set of all $k$-edge matchings of the complete bipartite graph~$K_{n,n}$.
\end{abstract}

\section{Introduction}

The ``Canada Day Theorem'' refers to the following curious combinatorial fact:
\begin{theorem}\label{thm:CanadaDay}
  Let $T$ be the $n\times n$ matrix with entries
  \begin{equation}
    T_{ij}=1+\sgn(i-j)
    =
    \begin{cases}
      0, & i<j,\\
      1, & i=j,\\
      2, & i>j,
    \end{cases}
  \end{equation}
  and let $X$ be an arbitrary symmetric $n\times n$ matrix.
  Then, for $1 \le k \le n$, the sum of the
  $k \times k$ principal minors of~$TX$ equals the
  sum of all $k\times k$ minors of~$X$ (principal and non-principal).
\end{theorem}

This theorem arose as an unexpected byproduct in our previous paper
\cite{hls} where we studied so-called ``peakon'' (peaked soliton)
solutions to a completely integrable nonlinear partial differential
equation discovered by V.~Novikov \cite{VNovikov}; a brief account of
this can be found in the appendix. The name of the theorem refers to its
``birthday'' (July 1, Canada's national holiday, in 2008).

A $k \times k$ minor in an $n \times n$ matrix $X = (x_{ij})$
is determined by a choice of $k$-element index sets $I$ and~$J$
(row and column indices, respectively)
from the set $[n] = \{ 1,2,\dots,n \}$.
We will use the notation
$\binom{[n]}{k}$ for the set of all $k$-element subsets of $[n]$,
and write minors as
\begin{equation}
  \minor{X_{IJ}}
  = \minor{X_{i_1 \dots i_k, j_1 \dots j_k}}
  = \det
  \begin{pmatrix}
    x_{i_1 j_1} & \dots & x_{i_1 j_k} \\
    \vdots & & \vdots \\
    x_{i_k j_1} & \dots & x_{i_k j_k} \\
  \end{pmatrix}
  ,
\end{equation}
for $I,J \in \binom{[n]}{k}$,
where $i_1 < i_2 < \dots < i_k$ and $j_1 < j_2 < \dots < j_k$ are the elements
of $I$ and~$J$ listed in increasing order.
A \emph{principal minor} is a minor with $I=J$.

The sum of the $k \times k$ principal minors of a matrix~$A$ is of course
very familiar as the coefficient of $(-1)^k z^{n-k}$ in the characteristic
polynomial $\det(zI-A)$.
The sum of \emph{all} $k \times k$ minors is much less frequently encountered;
one rare example in the literature is \cite{Okada},
but the results there do not seem directly related to Theorem~\ref{thm:CanadaDay},
since they deal with arbitrary matrices where symmetry plays no role.

\begin{example}
  In the case $n=3$ we have
  \begin{equation*}
    T =
    \begin{pmatrix}
      1 & 0 & 0 \\
      2 & 1 & 0 \\
      2 & 2 & 1
    \end{pmatrix}
    ,\qquad
    X =
    \begin{pmatrix}
      a & b & c \\
      b & d & e \\
      c & e & f
    \end{pmatrix}
    ,
  \end{equation*}
  and hence
  \begin{equation*}
    TX =
    \begin{pmatrix}
      a & b & c \\
      2a+b & 2b+d & 2c+e \\
      2a+2b+c & 2b+2d+e & 2c+2e+f
    \end{pmatrix}
    .
  \end{equation*}
  The sum of the principal $1 \times 1$ minors of $TX$ is just the trace,
  $(a+2b+2c)+(d+2e)+f$, which indeed equals the sum of all entries of~$X$
  (i.e., the sum of all $1 \times 1$ minors).
  The sum of the principal $2 \times 2$ minors of $TX$ is
  \begin{equation*}
    \begin{vmatrix}
      a & b \\
      2a+b & 2b+d \\
    \end{vmatrix}
    +
    \begin{vmatrix}
      a & c \\
      2a+2b+c & 2c+2e+f
    \end{vmatrix}
    +
    \begin{vmatrix}
      2b+d & 2c+e \\
      2b+2d+e & 2c+2e+f
    \end{vmatrix}
    ,
  \end{equation*}
  which, as can be easily verified, equals the sum of all $2 \times 2$ minors of $X$,
  \begin{equation*}
    \begin{vmatrix}
      a & b \\
      b & d \\
    \end{vmatrix}
    +
    \begin{vmatrix}
      a & c \\
      c & f \\
    \end{vmatrix}
    +
    \begin{vmatrix}
      d & e \\
      e & f \\
    \end{vmatrix}
    +
    \begin{vmatrix}
      a & c \\
      b & e \\
    \end{vmatrix}
    +
    \begin{vmatrix}
      b & d \\
      c & e \\
    \end{vmatrix}
    +
    \begin{vmatrix}
      b & e \\
      c & f \\
    \end{vmatrix}
    +
    \begin{vmatrix}
      a & b \\
      c & e \\
    \end{vmatrix}
    +
    \begin{vmatrix}
      b & c \\
      d & e \\
    \end{vmatrix}
    +
    \begin{vmatrix}
      b & c \\
      e & f \\
    \end{vmatrix}
    .
  \end{equation*}
  And, since $\det T=1$, the single $3 \times 3$ minor of~$TX$
  (which is of course principal)
  equals the single $3 \times 3$ minor of~$X$: $\det(TX) = \det X$.
\end{example}

\begin{example}
  For $n>3$, the cases $k=1$ and $k=n$ of the theorem are still easy to verify:
  for $k=1$ we have
  \begin{equation*}
    \tr(TX) = \tfrac12 \tr(TX) + \tfrac12 \tr((TX)^t)
    = \tr\bigr( \tfrac12 (T+T^t)X \bigr)
    = \sum_{i,j} x_{ij}
  \end{equation*}
  since $\tfrac12 (T+T^t)$ is the matrix with $1$ in every position,
  and for $k=n$ we have $\det(TX) = \det X$ as before.
  However, the intermediate cases $2 \le k \le n-1$ are more involved.
  (The reader might want to check the case $n=4$, $k=2$
  to become convinced of this.)
\end{example}

Of particular importance for the Canada Day Theorem are minors whose
row and column indices are \emph{interlacing}, meaning that
\begin{equation}
  \label{eq:interlacing}
    i_{1} \leq j_{1} \leq i_{2} \leq j_{2} \leq \dotsb \leq i_{k} \leq j_{k}.
\end{equation}
We abbreviate this condition as $I \le J$.
Given sets $I, J \in \binom{[n]}{k}$,
let $I' = I \setminus (I \cap J)$ and $J' = J \setminus (I \cap J)$,
and denote the cardinality of $I'$ and~$J'$ by
\begin{equation}
  \label{eq:p}
  p = p(I,J) = \card(I') = \card(J') = k - \card(I \cap J)
  .
\end{equation}
Note that $I$ and $J$ are interlacing if and only if $I'$ and $J'$
are \emph{strictly interlacing} (abbreviated $I' < J'$),
\begin{equation}
  \label{eq:strictly-interlacing}
    i'_{1} < j'_{1} < i'_{2} < j'_{2} < \dotsb < i'_{p} < j'_{p}.
\end{equation}
With this notation in place, we can state the more precise
version of the Canada Day Theorem that we are actually going to prove;
it says that the two sums are both equal to a third sum involving only
minors with interlacing index sets:
\begin{theorem}
  \label{thm:CDT-common-sum}
  Let $T$ and~$X$ be as in Theorem~\ref{thm:CanadaDay}, and let
  \begin{equation}
    \label{eq:common-sum}
    S = \sum_{\substack{I,J \in \binom{[n]}{k} \\[0.5ex] I \le J}} 2^{p(I,J)} \minor{X_{IJ}}
    .
  \end{equation}
  Then the following holds:
  \begin{enumerate}
  \item[(a)] The sum of the principal $k \times k$ minors of $TX$ equals~$S$.
  \item[(b)] The sum of all $k \times k$ minors of $X$ equals~$S$.
  \end{enumerate}
\end{theorem}

The following is a short outline of the paper. In Section~\ref{sec:II}
we use the Cauchy--Binet formula and the Lindstr\"om--Gessel--Viennot
path-counting lemma to prove the easier part of
Theorem~\ref{thm:CDT-common-sum}, namely part~(a),
which is true regardless of whether $X$ is symmetric or not.
Part~(b) is proved in Section~\ref{sec:Main}, and here it is crucial that
$X$ is a symmetric matrix.
We introduce a group of ``flips'' and study an
action of it on the set of all $k$-edge matchings of the complete
bipartite graph $K_{n,n}$.
The main technical point is dealt with in Lemma~\ref{lem:sign-when-flipped},
and this leads to the characterization of the orbit
structure of the group action presented in
Lemma~\ref{lem:even-endpoint-separation},
Corollary~\ref{cor:interlacing} and
Corollary~\ref{cor:non-interlacing}.
From these results it then follows that when expanding each $k \times k$
minor according to the definition of the determinant and adding everything
up, terms corresponding to orbits of a certain type (``non-interlacing'')
cancel out, while the other orbits (``interlacing'') give
contributions adding up to the sum~$S$.

It is fair to say that the proof given in the present paper is not
entirely different from the one in \cite{hls}. The central concept of
(open) clusters introduced in Section~\ref{sec:Main} has its
counterpart, namely linked pairs, in the original proof. However, the
organization of the proof, in particular the identification of the
group action on the set of all $k$-edge matchings as the main
underpinning for the Canada Day Theorem, is novel.

\section{The sum of the principal $k \times k$ minors of $TX$}
\label{sec:II}

In this section we prove part~(a) of Theorem~\ref{thm:CDT-common-sum}.
Recall the Cauchy--Binet formula for the minors of a matrix product
\cite[Ch.~I, \textsection~2]{gantmacher-matrixtheoryI}:
\begin{equation}
  \label{eq:Cauchy-Binet}
  \minor{(TX)_{AB}} = \sum_{I \in \binom{[1,n]}{k}} \minor{T_{AI}} \minor{X_{IB}},
  \qquad \text{for $A,B \in \textstyle \binom{[1,n]}{k}$}
  .
\end{equation}
Applying this with $A=B=J$ we can rewrite the sum of the principal
$k \times k$ minors of $TX$ as
\begin{equation}
  \label{eq:Cauchy-Binet-diagonal}
  \sum_{ J\in \binom{[n]}{k}} \minor{(TX)_{JJ}}
  = \sum_{ I,J\in \binom{[n]}{k}} \minor{T_{JI}} \minor{X_{IJ}}
  .
\end{equation}
Next, we need to compute the minors $\minor{T_{JI}}$.
Here the notion of interlacing index sets (as defined in the Introduction) enters.

\begin{lemma}
  \label{thm:TMinors}
  Let $T$ be defined as in Theorem~\ref{thm:CanadaDay}.
  Then, for
  $I, J \in \textstyle \binom{[1,n]}{k}$,
  \begin{equation}\label{eq:TMinors}
    \minor{T_{JI}} =
    \begin{cases}
      2^{p(I,J)}, & \text{if $I \leq J$}, \\
      0, & \text{otherwise},
    \end{cases}
  \end{equation}
  where $p(I,J) = k - \card(I \cap J)$.  
\end{lemma}

\begin{proof}
  The matrix $T$ is the path matrix of a planar directed graph of the form
  illustrated below in the case $n=4$
  (i.e., the entry $T_{ab}$ counts the number of paths -- zero, one, or two --
  from source vertex number~$a$
  on the left to sink vertex number~$b$ on the right):
  
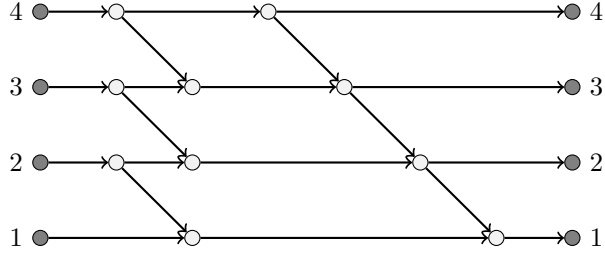
\begin{center}
    \begin{tikzpicture}[yscale=-1,sourcesink/.style={circle,draw=black!100,fill=black!50,inner sep = 0,minimum size=2mm},inter/.style={circle,draw=black!100,fill=black!5,inner sep=0,minimum size=2mm}]
      \node (isL)      at (-4,2) [sourcesink,label=180:1] {};
      \node (is+1L)    at (-4,1) [sourcesink,label=180:2] {};
      \node (is+1-1L)  at (-4,0) [sourcesink,label=180:3] {};
      \node (is+1+0L)  at (-4,-1) [sourcesink,label=180:4] {};
      \node (isR)      at (3,2) [sourcesink,label=0:1] {};
      \node (is+1R)    at (3,1) [sourcesink,label=0:2] {};
      \node (is+1-1R)  at(3,0) [sourcesink,label=0:3] {};
      \node (is+1+0R)  at(3,-1) [sourcesink,label=0:4] {};

      \node (is+1M1)   at(-3,1) [inter,draw] {};
      \node (is+1-1M1) at(-3,0) [inter,draw] {};
      \node (is+1+0M1) at(-3,-1) [inter,draw] {};
      \node (isM2)     at(-2,2) [inter,draw] {};
      \node (is+1M2)   at(-2,1) [inter,draw]  {};
      \node (is+1-1M2) at(-2,0) [inter,draw] {};

      \node (isM3)     at(2,2)   [inter,draw] {};
      \node (is+1M3)   at(1,1)   [inter,draw] {};
      \node (is+1-1M3) at(0,0)   [inter,draw] {};
      \node (is+1+0M3) at(-1,-1) [inter,draw] {};
      
      \begin{scope}[->,thick]
        \draw (isL) to (isM2);
        \draw (isM2) to (isM3);
        \draw (isM3) to (isR);
        \draw (is+1L) to (is+1M1);
        \draw (is+1M1) to (is+1M2);
        \draw (is+1M2) to (is+1M3);
        \draw (is+1M3) to (is+1R);
        \draw (is+1-1L) to (is+1-1M1);
        \draw (is+1-1M1) to (is+1-1M2);
        \draw (is+1-1M2) to (is+1-1M3);
        \draw (is+1-1M3) to (is+1-1R);
        \draw (is+1+0L) to (is+1+0M1);
        \draw (is+1+0M1) to (is+1+0M3);
        \draw (is+1+0M3) to (is+1+0R);
        
        \draw (is+1M1) to (isM2);
        \draw (is+1-1M1) to (is+1M2);
        \draw (is+1+0M1) to (is+1-1M2);
        
        \draw (is+1+0M3) to (is+1-1M3);
        \draw (is+1-1M3) to (is+1M3);
        \draw (is+1M3) to (isM3);
      \end{scope}
    \end{tikzpicture}
  \end{center}
  By the Lindstr\"om--Gessel--Viennot lemma
  \cite{KarlinMcGregor,lindstrom,gessel-viennot,Aigner},
  the minor $\minor{T_{JI}}$ equals
  the number of vertex-disjoint path families from sources
  indexed by~$J$ to sinks indexed by~$I$.

  Consider a vertex-disjoint path family
  from the source set $J = \{ 1\leq j_{1}<j_{2}<...<j_{k} \leq n \}$
  to the sink set $I = \{ 1 \leq i_{1}<i_{2}<...<i_{k} \leq n \}$.
  The planarity of of the graph allows only paths $j_{m} \to i_{m}$.
  Clearly $i_{m} \le j_{m}$, since no path can go upwards,
  and moreover $j_{m} \le i_{m+1}$, since otherwise both the path $j_{m} \to i_{m}$
  and the path $j_{m+1} \to i_{m+1}$ would have to pass through the vertex
  immediately to the left of sink vertex~$i_{m+1}$.
  Hence, if the interlacing condition $I \leq J$ is not satisfied,
  then there are no vertex-disjoint path families from $J$ to~$I$,
  and therefore $\minor{T_{JI}} = 0$ in this case.

  Suppose now that $I \leq J$.
  For each $m = 1,\dots,k$ we draw a rectangular window over the graph,
  with its upper left corner at source vertex~$j_{m}$
  and its lower right corner at sink vertex~$i_{m}$.
  Imagine trying to construct a vertex-disjoint path family from $J$ to~$I$.
  There will be only one possible path $j_{m} \to i_{m}$ if
  the $m$th window has its height $j_{m} - i_{m}$ equal to zero,
  or if it shares its top edge with another
  window; otherwise there will be two possible paths $j_{m} \to i_{m}$.
  Indeed, when the height is
  zero there is only one way to start and finish on the same level.
  When the top edge is shared, there is again only one path,
  since the path corresponding to the window directly above must use the
  vertex immediately to the left of sink vertex~$i_{m+1}$
  (which is on the same level as source vertex~$j_{m}$).
  Therefore, if we remove all windows which are of height zero or
  share their top edge with another window, then the number of
  possible vertex-disjoint path systems will be given by $2$ to the
  power of the number of remaining windows, in other words
  $2^{\card(I')}$ where $I'=I\setminus(I \cap J)$.
\end{proof}

Now, inserting the value of $\minor{T_{JI}}$ from this lemma into
equation \eqref{eq:Cauchy-Binet-diagonal}, we obtain
\begin{equation}
  \sum_{ J\in \binom{[n]}{k}} \minor{(TX)_{JJ}}
  = \sum_{\substack{I,J \in \binom{[n]}{k} \\[0.5ex] I \le J}} 2^{p(I,J)} \minor{X_{IJ}}
  = S
  ,
\end{equation}
which finishes the proof of part (a) of Theorem~\ref{thm:CDT-common-sum}.

\begin{remark}
  There are several ways to compute the minors $\minor{T_{JI}}$;
  see \cite{hls} for another argument using induction on~$n$.
  Note that the proof above is
  implicitly taking advantage of the factorization
  \begin{equation}\label{eq:TFactorization}
    T =
    \begin{pmatrix}
      1 & 0 & 0 & \dotsb & 0\\
      1 & 1 & 0 & \dotsb & 0\\
      0 & 1 & 1 & \ddots & \vdots\\
      \vdots & \ddots & \ddots & \ddots & 0\\
      0 & \dotsb & 0 & 1 & 1\\
    \end{pmatrix}
    \begin{pmatrix}
      1 & 0 & 0 & \dotsb & 0\\
      1 & 1 & 0 & \dotsb & 0\\
      1 & 1 & 1 & \ddots & \vdots\\
      \vdots & \ddots & \ddots & \ddots & 0\\
      1 & \dotsb & 1 & 1 & 1\\
    \end{pmatrix}.
  \end{equation}
\end{remark}

\section{The sum of all $k \times k$ minors of $X$}\label{sec:Main}

In this section we prove part~(b) of Theorem~\ref{thm:CDT-common-sum}.

\subsection{Minors and matchings}
\label{sec:minors-and-matchings}

By the definition of the determinant, we have
\begin{equation}
  \minor{X_{IJ}} = \det
  \begin{pmatrix}
    x_{i_1 j_1} & \dots & x_{i_1 j_k} \\
    \vdots & & \vdots \\
    x_{i_k j_1} & \dots & x_{i_k j_k} \\
  \end{pmatrix}
  = \sum_{\pi \in S_k} \sgn(\pi) x_{i_1 j_{\pi(1)}} \dotsm x_{i_k j_{\pi(k)}}
  ,
\end{equation}
where $S_k$ is the symmetric group of permutations on $k$ elements
(i.e., bijections $\pi \colon [k] \to [k]$).
A less index-heavy notation is obtained by using that
(for given $I$ and~$J$)
each permutation $\pi \colon [k] \to [k]$ corresponds to a unique
bijection $\tau \colon I \to J$ via $\tau(i_k) = j_{\pi(k)}$.
Setting $\sgn(\tau) = \sgn(\pi)$
and $\detterm{X}{\tau} = \prod_{i \in I} x_{i \tau(i)}$,
we simply get
\begin{equation}
  \minor{X_{IJ}} = \sum_{\tau \colon I \to J} \sgn(\tau) \detterm{X}{\tau}
\end{equation}
where the sum runs over all bijections from $I$ to~$J$.
We visualize such a bijection~$\tau$ as a bipartite graph;
for example, with $n=5$, the bijection
$\tau \colon  \{ 2,3,4 \} \to \{ 1,3,5 \}$
given by
$\tau(2)=3$, $\tau(3)=5$, $\tau(4)=1$
is drawn as
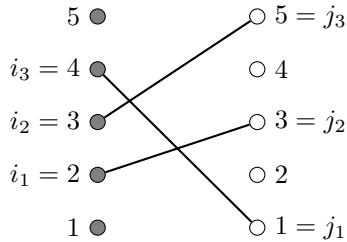
\begin{center}
  \begin{tikzpicture}[scale=0.7]
    \node (a1) at (0,1) [source,label={180:$    1$}] {};
    \node (a2) at (0,2) [source,label={180:$i_1=2$}] {};
    \node (a3) at (0,3) [source,label={180:$i_2=3$}] {};
    \node (a4) at (0,4) [source,label={180:$i_3=4$}] {};
    \node (a5) at (0,5) [source,label={180:$    5$}] {};
    \node (b1) at (3,1) [sink,label={0:$1=j_1$}] {};
    \node (b2) at (3,2) [sink,label={0:$2    $}] {};
    \node (b3) at (3,3) [sink,label={0:$3=j_2$}] {};
    \node (b4) at (3,4) [sink,label={0:$4    $}] {};
    \node (b5) at (3,5) [sink,label={0:$5=j_3$}] {};
    \begin{scope}[thick]
      \draw (a2) to (b3);
      \draw (a3) to (b5);
      \draw (a4) to (b1);
    \end{scope}
    \end{tikzpicture}
\end{center}
and it corresponds (for the given sets $I$ and~$J$)
to the permutation~$\pi \in S_3$ represented by the graph
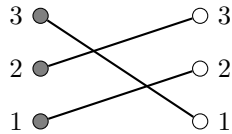
\begin{center}
  \begin{tikzpicture}[scale=0.7]
    \bipartite{3}
    \begin{scope}[thick]
      \draw (a1) to (b2);
      \draw (a2) to (b3);
      \draw (a3) to (b1);
    \end{scope}
    \end{tikzpicture}
\end{center}
The sign of~$\tau$ (and~$\pi$)
is $+1$ or~$-1$ depending on whether the crossing number of the graph
is even or odd. In this example there are two crossings, so
$\sgn(\tau) \detterm{X}{\tau} = +x_{23} x_{35} x_{41}$
(and this is one of the six terms in the $3 \times 3$ minor $\minor{X_{234,135}}$).

Note that when composing two bijections
$\tau_1 \colon I \to L$ and $\tau_2 \colon L \to J$,
the signs obey the same rule as for the corresponding permutations:
\begin{equation*}
  \sgn(\tau_2 \circ \tau_1) = \sgn(\tau_2) \sgn(\tau_1).
\end{equation*}

Choosing $k$-element sets $I$ and $J$ together with a bijection $\tau \colon I \to J$
is equivalent to choosing a $k$-edge matching
of the complete bipartite graph $K_{n,n}$:
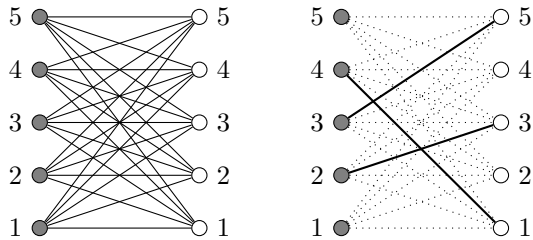
\begin{center}
  \begin{tikzpicture}[scale=0.7]
    \bipartite{5}
    \begin{scope}[thin]
      \foreach \n in {1,...,5}
        \foreach \m in {1,...,5} {
          \draw (a\n) to (b\m);
        }
    \end{scope}
  \end{tikzpicture}
  \qquad
  \begin{tikzpicture}[scale=0.7]
    \bipartite{5}
    \begin{scope}[thin,dotted]
      \foreach \n in {1,...,5}
        \foreach \m in {1,...,5} {
          \draw (a\n) to (b\m);
        }
    \end{scope}
    \begin{scope}[thick]
      \draw (a2) to (b3);
      \draw (a3) to (b5);
      \draw (a4) to (b1);
    \end{scope}
    \end{tikzpicture}
\end{center}
(Recall that a \emph{matching} of a graph $X=(V,E)$ with vertex set~$V$ and edge set~$E$
is a subset $F \subseteq E$ such that no two edges in $F$ share a common vertex.
One may equivalently think of the matching as being the subgraph $(V,F)$.
This habit of not distinguishing a graph from its edge set, when the underlying
vertex set is understood, will be used quite frequently.)

Fix $n$ and $k$,
and let $\mathcal{M} = \mathcal{M}_{n,k}$
denote the set of $k$-edge matchings of~$K_{n,n}$.
We will use the same symbol $\tau$ both for such a matching and
for the corresponding bijection,
and slightly abuse the language by speaking about matchings $\tau \colon I \to J$.
If $\tau(i) = j$ holds for the bijection~$\tau$,
then we say that the matching~$\tau$ contains an edge $i \to j$. 
(The graph $K_{n,n}$ is undirected, but the arrow notation is
convenient for distinguishing the left nodes labelled $1,\dots,n$ from the
right nodes also labelled $1,\dots,n$.)

When summing all $k \times k$ minors of an $n \times n$ matrix~$X$
and expanding each minor,
the whole sum turns into an alternating sum over all $k$-edge matchings of~$K_{n,n}$:
\begin{equation}
  \sum_{I,J \in \binom{[n]}{k}} \minor{X_{IJ}}
  = \sum_{I,J \in \binom{[n]}{k}} \sum_{\tau \colon I \to J} \sgn(\tau) \detterm{X}{\tau}
  = \sum_{\tau \in \mathcal{M}} \sgn(\tau) \detterm{X}{\tau}.
\end{equation}
In order to prove that this equals
\begin{equation*}
  \sum_{I \le J} 2^{p(I,J)} \minor{X_{IJ}}
  = \sum_{I \le J} \,\, \sum_{\tau \colon I \to J} 2^{p(I,J)} \sgn(\tau) \detterm{X}{\tau}
\end{equation*}
when $X$ is symmetric,
as claimed in Theorem~\ref{thm:CDT-common-sum},
we will introduce a certain abelian group $\mathcal{G}$ which acts on
$\mathcal{M}$ in a way which preserves the weights $\detterm{X}{\tau}$
but may change the signs $\sgn(\tau)$.
Then we compute the sum $\sum_{\tau \in \mathcal{M}}$ by adding the terms for
each group orbit separately and then summing over all orbits.
As we will see, each orbit containing a matching $\tau \colon I\rightarrow J$
with interlacing index sets $I\leq J$ contributes
$2^{p(I,J)}$ terms which all have the same sign,
while the other orbits contain equally many positive and negative terms
and therefore cancel out.

\subsection{Clusters of a matching}

Fix a matching $\tau \colon I \to J$
(viewed as a bipartite graph).
Temporarily add auxiliary horizontal edges $r \to r$
for all $r \in I \cap J$.
Split the resulting (multi)graph $\tilde{\tau}$ into connected components,
and then remove the auxiliary edges.
The remnants of the components of $\tilde{\tau}$ form a partition of the edges
of $\tau$ into what we will call \emph{clusters}.

Since $\tau$ is a matching, no vertex in $\tilde{\tau}$ can have degree greater
than two, and therefore the components of $\tilde{\tau}$ are paths.
A cluster is either called \emph{closed} or \emph{open},
depending on whether the corresponding component of $\tilde{\tau}$
is a closed or an open path.
The endpoints of an open path will also be said to be the endpoints of that open cluster.

\begin{example}
  \label{ex:clusters}
  Let $n=8$ and $k=7$.
  Here is a matching $\tau \colon I \to J$, where
  $I = \{ 1,2,3,4,5,6,8 \}$ and $J = \{ 1,2,4,6,7,8 \}$,
  together with its companion $\tilde{\tau}$ obtained by adding auxiliary horizontal edges
  $r \to r$ for $r = I \cap J = \{ 1,2,4,5,6,8 \}$:
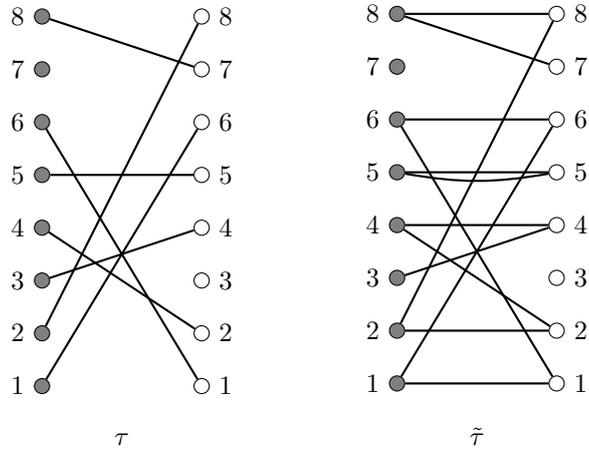
  \begin{center}
    \begin{tikzpicture}[scale=0.7]
      \bipartite{8}
      \begin{scope}[thick]
        \foreach \m/\n in {1/6,2/8,3/4,4/2,5/5,6/1,8/7}
          \draw (a\m) to (b\n);
      \end{scope}
      \node at (1.5,0) {$\tau$};
    \end{tikzpicture}
    \qquad\qquad
    \begin{tikzpicture}[scale=0.7]
      \bipartite{8}
      \begin{scope}[thick]
        \foreach \m/\n in {1/6,2/8,3/4,4/2,5/5,6/1,8/7}
          \draw (a\m) to (b\n);
        \foreach \n in {1,2,4,6,8}
          \draw (a\n) to (b\n);
        \draw (a5) to [out=-10,in=190] (b5);
      \end{scope}
      \node at (1.5,0) {$\tilde{\tau}$};
    \end{tikzpicture}
  \end{center}
  There are three connected components in~$\tilde{\tau}$
  (one open path with endpoints $3 \in I$ and~$7 \in J$,
  and two closed paths):
  
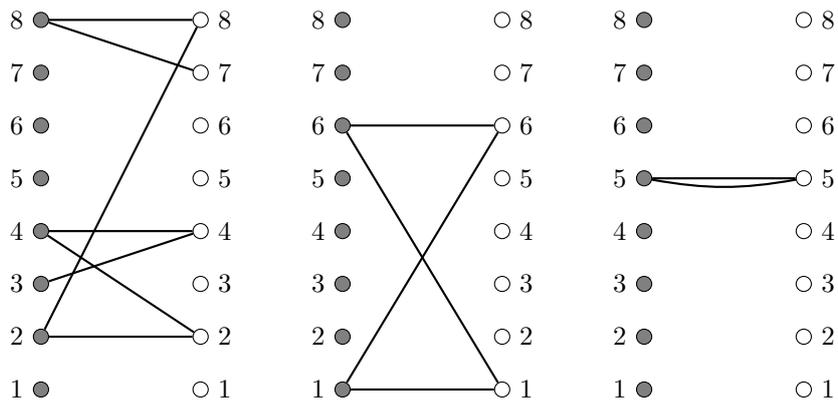
\begin{center}
    \begin{tikzpicture}[scale=0.7]
      \bipartite{8}
      \begin{scope}[thick]
        \foreach \m/\n in {2/8,3/4,4/2,8/7,2/2,4/4,8/8}
          \draw (a\m) to (b\n);
      \end{scope}
    \end{tikzpicture}
    \qquad
    \begin{tikzpicture}[scale=0.7]
      \bipartite{8}
      \begin{scope}[thick]
        \foreach \m/\n in {1/6,6/1,1/1,6/6}
          \draw (a\m) to (b\n);
      \end{scope}
    \end{tikzpicture}
    \qquad
    \begin{tikzpicture}[scale=0.7]
      \bipartite{8}
      \begin{scope}[thick]
        \draw (a5) to (b5);
        \draw (a5) to [out=-10,in=190] (b5);
      \end{scope}
    \end{tikzpicture}
  \end{center}
  Removing the auxiliary edges,
  we obtain the three clusters of the matching~$\tau$
  (one open and two closed):
  
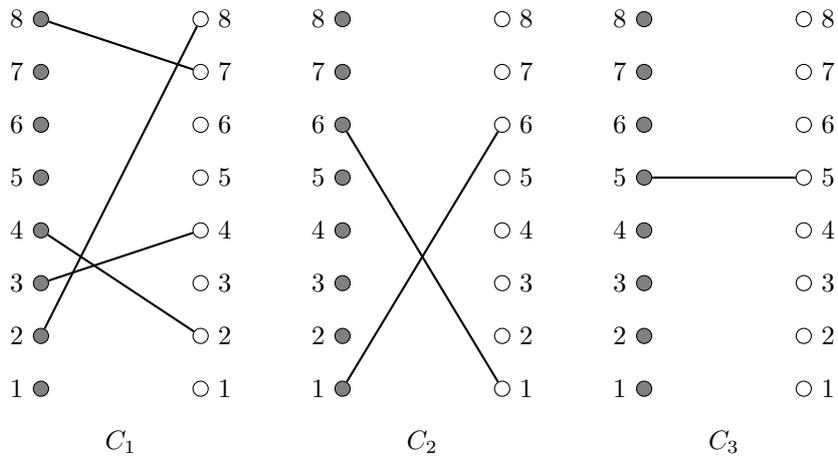
\begin{center}
    \begin{tikzpicture}[scale=0.7]
      \bipartite{8}
      \begin{scope}[thick]
        \foreach \m/\n in {2/8,3/4,4/2,8/7}
          \draw (a\m) to (b\n);
      \end{scope}
      \node at (1.5,0) {$C_1$};
    \end{tikzpicture}
    \qquad
    \begin{tikzpicture}[scale=0.7]
      \bipartite{8}
      \begin{scope}[thick]
        \foreach \m/\n in {1/6,6/1}
          \draw (a\m) to (b\n);
      \end{scope}
      \node at (1.5,0) {$C_2$};
    \end{tikzpicture}
    \qquad
    \begin{tikzpicture}[scale=0.7]
      \bipartite{8}
      \begin{scope}[thick]
        \draw (a5) to (b5);
      \end{scope}
      \node at (1.5,0) {$C_3$};
    \end{tikzpicture}
  \end{center}
\end{example}

\begin{definition}[Endpoint separation]
  Let $\tau \colon I \rightarrow J$ be a matching and $C$ one of its open clusters.  
  Consider the list consisting of all numbers in $I\cup J$, sorted in ascending order.
  (If $i_{m} = j_{m'}$ for some $m$ and $m'$,
  then include that number twice in the list.)
  Let the \emph{endpoint separation} $\separation{C}$ of the cluster~$C$
  be the number of entries in this sorted list lying strictly between
  the numbers labelling the two endpoints of~$C$.
  For closed clusters~$C$, we set $\separation{C}=0$.
\end{definition}

\begin{remark} 
Observe that the endpoint separation is a property of both 
the matching and the cluster, and not just of the cluster itself.
\end{remark} 

\begin{example}
  \label{ex:endpoint-separation}
  The open cluster $C_1$ in Example~\ref{ex:clusters},
  with endpoints labelled $3$ and~$7$,
  has endpoint separation $\separation{C_1} = 6$,
  since there are six numbers ($4$, $4$, $5,5$, $6,6$) strictly between the
  $3$ and~$7$ in the list $(1,1,2,2,3,4,4,5,5,6,6, 7,8,8)$.
  For the closed clusters,
  we have $\separation{C_2} = \separation{C_3} = 0$ by definition.
\end{example}

\subsection{The group of flips}

Given~$n$, let $\mathcal{G} = \mathcal{G}_n$
be the abelian group obtained by taking the direct sum
of $\binom{n}{2}$ copies of the two-element group $\Z/(2)$
(one copy for each pair $(i,j)$ with $1 \le i < j \le n$).
We define an action of $\mathcal{G}$ on the set $\mathcal{M}$
of $k$-edge matchings of~$K_{n,n}$ as follows:
\begin{itemize}
\item Let $f_{ij} = (0,\dots,0,1,0,\dots,0) \in \mathcal{G}$,
  where the~$1$ is in the $(i,j)$th copy of $\Z/(2)$.
  The elements $\{ f_{ij} \}_{i<j}$ generate $\mathcal{G}$,
  so it is enough to define how they act.
\item If there is an open cluster $C$ in the matching $\tau \in \mathcal{M}$
  containing one of the edges $i \to j$ and $j \to i$
  (note that an open cluster cannot contain both),
  then let $f_{ij} \smallbullet \tau$ be the matching obtained by
  \emph{flipping} the whole cluster~$C$, i.e., replacing each edge $a \to b$
  in~$C$ by $b \to a$ (and leaving all other edges in $\tau$ as they are).
\item Otherwise, let $f_{ij}$ do nothing: $f_{ij} \smallbullet \tau = \tau$.
\end{itemize}
It is straightforward to verify that this really defines a group action as claimed,
since flips commute and flipping the same cluster twice is the identity
transformation.
Note that it is the flipping of a whole cluster (rather than just an
individual edge) that ensures that $f_{ij} \smallbullet \tau$ is
still a matching.

\begin{example}
  \label{ex:flip}
  When $f_{28}$ acts on the matching $\tau$ of Example~\ref{ex:clusters},
  the open cluster $C_1$ (here drawn dotted) is flipped,
  since it is open (with endpoints $3$ and~$7$)
  and contains the edge $2 \to 8$:
  
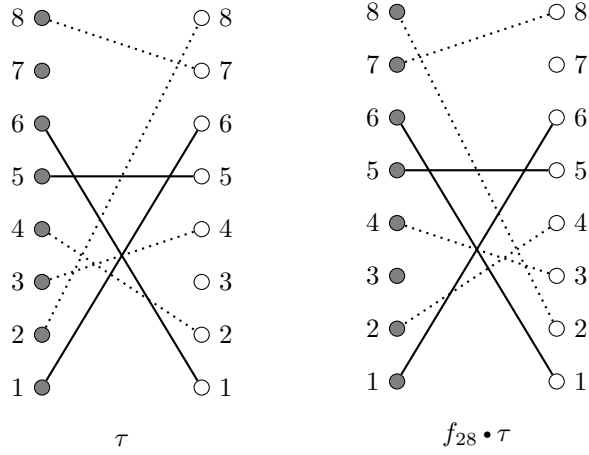
\begin{center}
    \begin{tikzpicture}[scale=0.7]
      \bipartite{8}
      \begin{scope}[thick]
        \foreach \m/\n in {1/6,5/5,6/1}
          \draw (a\m) to (b\n);
        \foreach \m/\n in {2/8,3/4,4/2,8/7}
          \draw[dotted] (a\m) to (b\n);
      \end{scope}
      \node at (1.5,0) {$\tau$};
    \end{tikzpicture}
    \qquad\qquad
    \begin{tikzpicture}[scale=0.7]
      \bipartite{8}
      \begin{scope}[thick]
        \foreach \m/\n in {1/6,5/5,6/1}
          \draw (a\m) to (b\n);
        \foreach \m/\n in {2/8,3/4,4/2,8/7}
          \draw[dotted] (b\m) to (a\n);
      \end{scope}
      \node at (1.5,0) {$f_{28} \smallbullet \tau$};
    \end{tikzpicture}
  \end{center}
\end{example}

\begin{lemma}
  \label{lem:weight-when-flipped}
  If the matrix~$X$ is symmetric,
  then $\detterm{X}{g \smallbullet \tau} = \detterm{X}{\tau}$
  for all flips $g \in \mathcal{G}$ and all matchings $\tau \in \mathcal{M}$.
\end{lemma}

\begin{proof}
  This is immediate, since replacing the edge $a \to b$ by the edge $b \to a$
  in the matching~$\tau \colon I \to J$
  corresponds to replacing the matrix entry $X_{ab}$ by $X_{ba}$ in the
  product~$\detterm{X}{\tau} = \prod_{i \in I} X_{i \tau(i)}$.
\end{proof}

\begin{example}
  \label{ex:weight-flipped}
  The flip in Example~\ref{ex:flip} corresponds to changing
  \begin{equation*}
    \detterm{X}{\tau} = X_{16} X_{28} X_{34} X_{42} X_{55} X_{61} X_{87}
  \end{equation*}
  into
  \begin{equation*}
    \detterm{X}{f_{28} \smallbullet \tau} = X_{16} X_{82} X_{43} X_{24} X_{55} X_{61} X_{78}
    .
  \end{equation*}
\end{example}

A less trivial aspect of the group action is the relation between
$\sgn(g \smallbullet \tau)$ and~$\sgn(\tau)$,
and this is where the endpoint separation and the
interlacing condition $I \le J$ will be of importance.

\begin{lemma}
  \label{lem:sign-when-flipped} 
  If the cluster $C$ is flipped when $f_{ij}$ acts on $\tau$, then
  \begin{equation}
    \label{eq:clusterflip}
    \sgn(f_{ij} \smallbullet \tau) = (-1)^{\separation{C}} \sgn(\tau).
  \end{equation}
\end{lemma}

\begin{proof}
  As remarked in Section~\ref{sec:minors-and-matchings},
  the composition of matchings $\tau_1 \colon I \to L$ and $\tau_2 \colon L \to J$
  (thought of as functions) satisfies
  $\sgn(\tau_2 \circ \tau_1) = \sgn(\tau_2) \sgn(\tau_1)$.
  We will split $\tau \colon I \to J$ into such a composition
  where $\tau_1$ deals with the edges in the cluster
  and $\tau_2$ with the remaining edges.

  Consider an open cluster~$C$ in~$\tau$ with endpoints
  $a \in I' = I \setminus (I \cap J)$
  and~$b \in J' = J \setminus (I \cap J)$:
  \begin{equation*}
    \tau(a) = c_1
    ,\quad
    \tau(c_1) = c_2
    ,\quad
    \tau(c_2) = c_3
    ,\quad
    \dots
    ,\quad
    \tau(c_{m-1}) = c_{m}
    ,\quad
    \tau(c_m) = b
    .
  \end{equation*}
  We let $K = \{ c_1, c_2, \dots, c_m \} \subset I \cap J$ and define
  $\tau_1 \colon I \to I+b-a$ and $\tau_2 \colon I+b-a \to J$ by
  \begin{equation}
    \tau_1(x) =
    \begin{cases}
      \tau(x), & \text{if $x \in K+a$}
      ,\\
      x, & \text{if $x \in I \setminus (K+a)$}
      ,
    \end{cases}
    \qquad
    \tau_2(x) =
    \begin{cases}
      x, & \text{if $x \in K+b$}
      ,\\
      \tau(x), & \text{if $x \in I \setminus (K+b)$}
      .
    \end{cases}
  \end{equation}
  (For readability,
  we have written $I+b-a$ instead of $(I \cup \{ b \}) \setminus \{ a \}$
  and $K+a$ instead of $K \cup \{ a \}$.)
  These definitions imply that $\tau = \tau_2 \circ \tau_1$.
  Schematically, with dashed arrows indicating the mapping
  $x \mapsto x$ and solid arrows indicating $x \mapsto \tau(x)$:
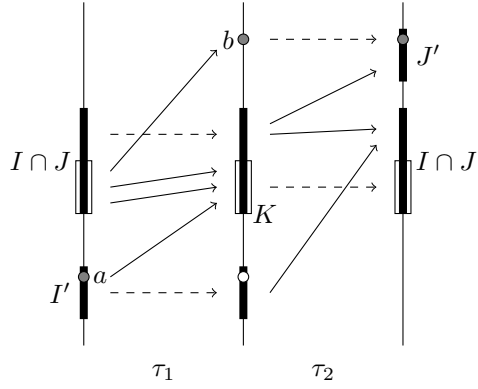
  \begin{center}
    \begin{tikzpicture}[scale=0.7]
      \draw (0,0) -- +(0,6.5);
      \begin{scope}[line width=3pt]
        \draw (0,0.5) to node[left] {$I'$} +(0,1);
        \draw (0,2.5) to node[left] {$I \cap J$} +(0,2);
      \end{scope}
      \draw (-0.15,2.5) rectangle node[right] {$ $} +(0.3,1);
      \draw[fill=gray] (0,1.3) circle (0.1) node[right] {$a$};

      \begin{scope}[xshift=3cm]
        \draw (0,0) -- +(0,6.5);
        \begin{scope}[line width=3pt]
          \draw (0,0.5) to node[left] {$ $} +(0,1);
          \draw (0,2.5) to node[left] {$ $} +(0,2);
        \end{scope}
        \draw (-0.15,2.5) rectangle +(0.3,1);
        \draw (0.4,2.5) node {$K$};
        \draw[fill=white] (0,1.3) circle (0.1) node[right] {$ $};
        \draw[fill=gray] (0,5.8) circle (0.1) node[left] {$b$};
      \end{scope}

      \begin{scope}[xshift=6cm]
        \draw (0,0) -- +(0,6.5);
        \begin{scope}[line width=3pt]
          \draw (0,5) to node[right] {$J'$} +(0,1);
          \draw (0,2.5) to node[right] {$I \cap J$} +(0,2);
        \end{scope}
        \draw (-0.15,2.5) rectangle node[right] {$ $} +(0.3,1);
        \draw[fill=gray] (0,5.8) circle (0.1) node[left] {$ $};
      \end{scope}

      \begin{scope}[->,dashed,xshift=0.5cm]
        \draw (0,1) to +(2,0);
        \draw (0,4) to +(2,0);
      \end{scope}

      \begin{scope}[xshift=0.5cm,yshift=2.5cm]
        \begin{scope}[->]
          \draw (0,-1.2) -- (2,0.2);
          \draw (0,0.2) -- (2,0.5);
          \draw (0,0.5) -- (2,0.8);
          \draw (0,0.8) -- (2,3.1);
        \end{scope}
      \end{scope}
      \draw (1.5,-0.5) node {$\tau_1$};

      \begin{scope}[xshift=3.5cm]
        \begin{scope}[->]
          \draw[dashed] (0,3) to +(2,0);
          \draw (0,4.2) to (2,5.2);
          \draw (0,4) -- (2,4.1);
          \draw (0,1) to (2,3.8);
          \draw[dashed] (0,5.8) -- +(2,0);
        \end{scope}
        \draw (1,-0.5) node {$\tau_2$};
      \end{scope}
    \end{tikzpicture}
  \end{center}
  The result of flipping the cluster~$C$ is
  $f_{ij} \smallbullet \tau = \tau_3 \circ \tau_1^{-1}$,
  where $\tau_3 \colon I \to J+a-b$ differs from $\tau_2$
  in having the edge $a \to a$ instead of the edge $b \to b$:
  \begin{center}
    \begin{tikzpicture}[scale=0.7]
      \draw (0,0) -- +(0,6.5);
      \begin{scope}[line width=3pt]
        \draw (0,0.5) to node[left] {$I'-a$} +(0,1);
        \draw (0,2.5) to node[left] {$I \cap J$} +(0,2);
      \end{scope}
      \draw (-0.15,2.5) rectangle node[right] {$ $} +(0.3,1);
      \draw[fill=white] (0,1.3) circle (0.1) node[right] {$ $};
      \draw[fill=gray] (0,5.8) circle (0.1) node[left] {$b$};

      \begin{scope}[xshift=3cm]
        \draw (0,0) -- +(0,6.5);
        \begin{scope}[line width=3pt]
          \draw (0,0.5) to node[left] {$ $} +(0,1);
          \draw (0,2.5) to node[left] {$ $} +(0,2);
        \end{scope}
        \draw (-0.15,2.5) rectangle +(0.3,1);
        \draw (0.4,2.5) node {$K$};
        \draw[fill=gray] (0,1.3) circle (0.1) node[right] {$ $};
      \end{scope}

      \begin{scope}[xshift=6cm]
        \draw (0,0) -- +(0,6.5);
        \begin{scope}[line width=3pt]
          \draw (0,5) to node[right] {$J'-b$} +(0,1);
          \draw (0,2.5) to node[right] {$I \cap J$} +(0,2);
        \end{scope}
        \draw (-0.15,2.5) rectangle node[right] {$ $} +(0.3,1);
        \draw[fill=white] (0,5.8) circle (0.1) node[left] {$ $};
        \draw[fill=gray] (0,1.3) circle (0.1) node[right] {$a$};
      \end{scope}

      \begin{scope}[->,dashed,xshift=0.5cm]
        \draw (0,1) to +(2,0);
        \draw (0,4) to +(2,0);
      \end{scope}

      \begin{scope}[xshift=0.5cm,yshift=2.5cm]
        \begin{scope}[->]
          \draw (0,0.2) -- (2,-1.2);
          \draw (0,0.5) -- (2,0.2);
          \draw (0,0.8) -- (2,0.5);
          \draw (0,3.1) -- (2,0.8);
        \end{scope}
      \end{scope}
      \draw (1.5,-0.5) node {$\tau_1^{-1}$};

      \begin{scope}[xshift=3.5cm]
        \begin{scope}[->]
          \draw[dashed] (0,3) to +(2,0);
          \draw (0,4.2) to (2,5.2);
          \draw (0,4) -- (2,4.1);
          \draw (0,1) to (2,3.8);
          \draw[dashed] (0,1.3) -- +(2,0);
        \end{scope}
        \draw (1,-0.5) node {$\tau_3$};
      \end{scope}
    \end{tikzpicture}
  \end{center}
  We have $\sgn(\tau_3) = (-1)^{\separation{C}} \sgn(\tau_2)$,
  since if we imagine detaching the edge $b \to b$
  from its vertices and continuously sliding it to the position
  $a \to a$,
  the crossing number of the matching changes by one each time either
  the left or the right end of of that edge moves past a matched vertex,
  and by the definition of endpoint
  separation there are exactly $\separation{C}$ matched vertices on the levels
  strictly between $a$ and~$b$.
  Together with $\sgn(\tau_1^{-1}) = \sgn(\tau_1)$,
  this proves \eqref{eq:clusterflip}.
\end{proof}

Let us call a matching $\tau \colon I \to J$ interlacing
if the sets $I$ and~$J$ are interlacing, $I \le J$ (cf. \eqref{eq:interlacing}).
The orbits of the interlacing matchings
(under the action of the flip group~$\mathcal{G}$)
will be called interlacing orbits,
and all other orbits non-interlacing.

\begin{lemma}
  \label{lem:even-endpoint-separation}
  A matching $\sigma$ belongs to an interlacing orbit
  if and only if
  every cluster~$C$ in~$\sigma$ has \textbf{even} endpoint separation $\separation{C}$.
\end{lemma}

\begin{proof}
  It is clear from the definitions that flipping a cluster does not
  change its endpoint separation.
  Closed clusters always have even endpoint separation by
  definition (namely, zero).
  If $C$ is an open cluster in an interlacing matching~$\tau \colon I \to J$,
  with endpoints $i_s \in I' = I \setminus (I \cap J)$ and $j_t \in J' =J\setminus (I \cap J)$,
  then in the case $i_s < j_t$ we find from the interlacing condition that
  \begin{equation*}
    \dotsb \leq j_{s-1} < i_{s} < \underbrace{j_{s} \leq \dotsb \leq i_{t}}_{\text{$2(t-s)$ elements}} < j_{t} < i_{t+1} \leq \dotsb 
    ,
  \end{equation*}
  so that $\separation{C} = 2(t-s)$ is even, and the case $i_s > j_t$ is similar.
  It follows that if $\sigma$ belongs to the orbit of an interlacing matching,
  then every cluster in~$\sigma$ has even endpoint separation~$\separation{C}$.

  Conversely, given a matching $\sigma \colon A \to B$ all of whose clusters
  have even endpoint separation,
  we form interlacing sets $I \le J$ by sorting the list of numbers
  $(a_1,\dots,a_k,b_1,\dots,b_k)$ in ascending order and labelling
  the elements of the sorted list as $(i_1, j_1, i_2, j_2, \dots, i_k, j_k)$
  (note that $I \cap J = A \cap B$).
  If $a \in A$ and $b \in B$ are the endpoints of an open cluster $C$ in~$\sigma$,
  then $(a,b) \in I \times J$ or $J \times I$ (rather than $I \times I$ or $J \times J$)
  because of the assumption that $\separation{C}$ is even.
  Flipping those open clusters in $\sigma$ whose endpoints belong to $J \times I$ produces
  an interlacing matching $\tau \colon I \to J$ whose orbit $\sigma$ belongs to.
\end{proof}

\begin{corollary}
  \label{cor:interlacing}
  All matchings in a given interlacing orbit have the same sign.
\end{corollary}

\begin{proof}
  Combine Lemmas \ref{lem:sign-when-flipped} and~\ref{lem:even-endpoint-separation}.
\end{proof}

\begin{corollary}
  \label{cor:non-interlacing}
  In a non-interlacing orbit, there are equally many matchings of each sign.
\end{corollary}

\begin{proof}
  By Lemma~\ref{lem:even-endpoint-separation}, each matching in a non-interlacing
  orbit has at least one open cluster with odd endpoint separation;
  among those clusters we single out the unique one with the property that its
  lowest-numbered node is smaller than that of the other ones.
  The operation of flipping that cluster is a sign-reversing
  involution pairing up the matchings in the orbit.
  (It's sign-reversing by Lemma~\ref{lem:sign-when-flipped},
  and it's an involution since the lowest-numbered node will still be
  the lowest-numbered node after the flip.)
\end{proof}

\begin{lemma}
  \label{lem:orbit-cardinality}
  (a)~The orbit $\mathcal{G} \smallbullet \tau$ of a matching $\tau \colon I \to J$ contains $2^{p(I,J)}$ elements.
  (b)~Each interlacing orbit contains exacly one interlacing matching.
\end{lemma}

\begin{proof}
  Recall that $p(I,J) = \card(I') = \card(J')$ where $I' = I \setminus
  (I \cap J)$ and $J' = J \setminus (I \cap J)$. Each element of $I'$
  is the endpoint of an open cluster whose other endpoint belongs to
  $J'$ (by the definition of an open cluster), while all other
  clusters are closed. Thus there are exactly as many open clusters
  in~$\tau$ as there are elements in~$I'$ and~$J'$, and each of these
  $p(I,J)$ open clusters can be flipped independently of the others.
  This proves part~(a). To prove part (b) we consider two cases. If an
  interlacing matching $\tau$ has only closed clusters then by (a)
  $\card(\mathcal{G} \smallbullet \tau)=1$, hence (b) holds. If, on
  the other hand, an interlacing matching $\tau$ has at least one open
  cluster then flipping any open cluster destroys the interlacing
  property. This follows from the fact that $I \cup J$
  and $I \cap J$ are invariant under flipping, and that there is only
  one way of constructing interlacing sets with given union and
  intersection.
\end{proof}

Now we only have to put the pieces together:

\begin{proof}[Proof of part (b) of Theorem~\ref{thm:CDT-common-sum}]
  Let $\mathcal{M}^* \subset \mathcal{M}$ denote the union of the
  interlacing orbits. Then, referring in brackets to the relevant
  Lemmas and Corollaries above, we have
  \begin{align*}
    \sum_{I, J} \minor{X_{IJ}}
    &= \sum_{I, J} \sum_{\tau \colon I \to J} \sgn(\tau) \detterm{X}{\tau}
    && \text{(def. of determinant)} \\
    &= \sum_{\tau \in \mathcal{M}} \sgn(\tau) \detterm{X}{\tau}
    && \\
    &= \sum_{\tau \in \mathcal{M}^*} \sgn(\tau) \detterm{X}{\tau} + \sum_{\tau \notin \mathcal{M}^*} \sgn(\tau) \detterm{X}{\tau}
    && \\
    &= \sum_{I \le J}  \sum_{\tau \colon I \to J} \sum_{\sigma \in \mathcal{G} \smallbullet \tau} \sgn(\sigma) \detterm{X}{\sigma}
    && \text{(def. of $\mathcal{M}^*$, \ref{lem:orbit-cardinality}b)} \\
    & \quad + 0
    && \text{(\ref{lem:weight-when-flipped}, \ref{cor:non-interlacing})}
    \\[1ex]
    &= \sum_{I \le J}  \sum_{\tau \colon I \to J} 2^{p(I,J)} \sgn(\tau) \detterm{X}{\tau}
    && \text{(\ref{lem:weight-when-flipped}, \ref{cor:interlacing}, \ref{lem:orbit-cardinality}a)} \\
    &= \sum_{I \le J} 2^{p(I,J)} \left( \sum_{\tau \colon I \to J} \sgn(\tau) \detterm{X}{\tau} \right)
    && \\
    &= \sum_{I \le J} 2^{p(I,J)} \minor{X_{IJ}}
    && \text{(def. of determinant)}
    .
  \end{align*}
\end{proof}

\section{Acknowledgments}
This work was supported by the Swedish Research Council
(\textit{Vetenskaps\-r{\aa}det}) [VR 2010-5822 to H. Lundmark]
and the Natural Sciences and Engineering Research Council of Canada
[NSERC USRA 414987 to D. Gomez and NSERC 163953 to J. Szmigielski].

\appendix

\section{Appendix}

\newcommand{\peakonplotleft}{-2.5}
\newcommand{\peakonplotright}{16}
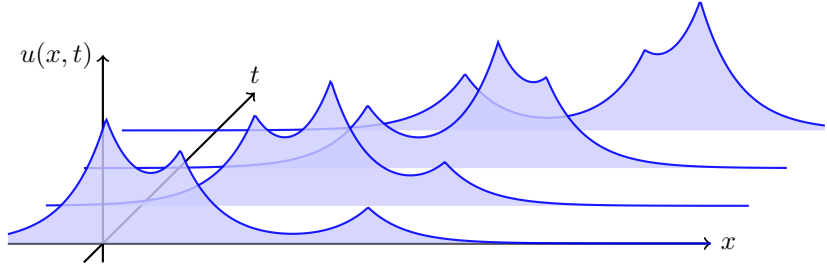
\begin{figure}
  \begin{center}
    \begin{tikzpicture}[samples=500,domain=\peakonplotleft:\peakonplotright,xscale=0.5,yscale=0.5]
      \draw[->,thick] (\peakonplotleft,0) -- (\peakonplotright,0) node[right] {$x$};
      \draw[->,thick] (-0.5,-0.5) -- (4,4) node[above] {$t$};

      \fill[blue!20, opacity=0.8, shift={(3,3)}] (\peakonplotleft,0) -- plot[id=peakonplot3] function{1.48*exp(-abs(x-6.53)) + 1.40*exp(-abs(x-11.26)) + 3.11*exp(-abs(x-12.72))} -- (\peakonplotright,0) -- cycle;
      \draw[blue, opacity=0.9, thick, shift={(3,3)}] plot[id=peakonplot3] function{1.48*exp(-abs(x-6.53)) + 1.40*exp(-abs(x-11.26)) + 3.11*exp(-abs(x-12.72))};

      \fill[blue!20, opacity=0.8, shift={(2,2)}] (\peakonplotleft,0) -- plot[id=peakonplot2] function{1.56*exp(-abs(x-4.97)) + 2.84*exp(-abs(x-8.40)) + 1.60*exp(-abs(x-9.67))} -- (\peakonplotright,0) -- cycle;
      \draw[blue, opacity=0.9, thick, shift={(2,2)}] plot[id=peakonplot2] function{1.56*exp(-abs(x-4.97)) + 2.84*exp(-abs(x-8.40)) + 1.60*exp(-abs(x-9.67))};

      \fill[blue!20, opacity=0.8, shift={(1,1)}] (\peakonplotleft,0) -- plot[id=peakonplot1] function{2*exp(-abs(x-3)) + 3*exp(-abs(x-5)) + exp(-abs(x-8))} -- (\peakonplotright,0) -- cycle;
      \draw[blue, opacity=0.9, thick, shift={(1,1)}] plot[id=peakonplot1] function{2*exp(-abs(x-3)) + 3*exp(-abs(x-5)) + exp(-abs(x-8))};

      \draw[->,thick] (0,-0.5) -- +(0,5.5) node[left] {$u(x,t)$};

      \fill[blue!20, opacity=0.8, shift={(0,0)}] (\peakonplotleft,0) -- plot[id=peakonplot0] function{3.02*exp(-abs(x-0.085)) + 2.04*exp(-abs(x-2.04)) + 0.94*exp(-abs(x-6.98))} -- (\peakonplotright,0) -- cycle;
      \draw[blue, opacity=0.9, thick, shift={(0,0)}] plot[id=peakonplot0] function{3.02*exp(-abs(x-0.085)) + 2.04*exp(-abs(x-2.04)) + 0.94*exp(-abs(x-6.98))};
    \end{tikzpicture}
  \end{center}
  \caption{A three-peakon solution of the Camassa--Holm equation.}
  \label{fig:PeakonExample}
\end{figure}

The Canada Day Theorem first arose in the context of nonlinear partial
differential equations in \cite{hls}. In this appendix we describe briefly the
subject of that paper and how it resulted in the formulation of the
theorem. The nonlinear partial differential equation
\begin{equation}\label{eq:Novikov}
  u_t - u_{xxt} + 4 u^2 u_x = 3 u u_x u_{xx} + u^2 u_{xxx},
\end{equation}
where $u=u(x,t)$, $u_x=\frac{\partial u}{\partial x}(x,t)$,
$u_t=\frac{\partial u}{\partial t}(x,t)$, etc., was derived by V.
Novikov \cite{VNovikov} as part of a classification of generalized
Camassa--Holm-type equations possessing infinite hierarchies of higher
symmetries. The Camassa--Holm equation
\begin{equation}\label{eq:CH}
  u_t - u_{xxt} + 3uu_x = 2u_x u_{xx} + uu_{xxx},
\end{equation}
which was formulated by Camassa and Holm in \cite{ch} as an integrable
shallow water equation (with $u$ denoting a horizontal velocity component),
admits a particularly
interesting class of explicit weak solutions which capture phenomena
such as the collision and breakdown of waves. Miraculously,
many of these features turn out to be intrinsically connected to
classical problems in analysis, like Stieltjes continued fractions
\cite{bss-stieltjes}, or the moment problem \cite{bss-moment}. These
weak explicit solutions are called \emph{peakons} (short for
\emph{peaked solitons}) because of the characteristic $e^{-\abs{x}}$ shape
of the waves.
Multipeakon solutions are formed by the superposition of $n$ peakons,
\begin{equation}\label{eq:PeakonAnsatz}
  u(x,t) = \sum_{i=1}^n m_i(t) \, e^{-\abs{x-x_i(t)}},
\end{equation}
with a suitable time dependence in the amplitudes~$m_i(t)$ and positions~$x_i(t)$.
Observe that since these solutions are not differentiable everywhere they only
satisfy the PDE \eqref{eq:CH} in a certain weak sense. The evolution of a three-peakon
solution over time is illustrated in Figure~\ref{fig:PeakonExample}.
One of the objectives of \cite{hls} was to find explicit formulas
for the positions and amplitudes of the peaks for the multipeakon solution of
Novikov's equation~\eqref{eq:Novikov}.
After substituting the ansatz \eqref{eq:PeakonAnsatz} into
\eqref{eq:Novikov}, taking into account the weak nature of the
solutions, one obtains $2n$ ordinary differential equations
which govern the time dependence of the positions and amplitudes:
\begin{equation}\label{eq:NovikovODE}
  \begin{split}
    \dot x_k &
    = \left( \sum_{i=1}^n m_i \, e^{-\abs{x_k-x_i}} \right)^2,
    \\
    \dot m_k &= m_k \, \left( \sum_{i=1}^n m_i \, e^{-\abs{x_k-x_i}} \right)
    \left( \sum_{j=1}^n m_j \, \sgn(x_k-x_j) \, e^{-\abs{x_k-x_j}} \right).
  \end{split}
\end{equation}
(By definition, $\sgn(0) = 0$ here.)
These equations were already stated in
\cite{hone-wang-cubic-nonlinearity} where it was shown that they
constitute a Hamiltonian system, and one of the main results of
\cite{hls} was that they are also Arnol{\cprime}d--Liouville integrable \cite{Arnold},
meaning that there exist $n$ functionally independent and Poisson commuting
constants of motion $H_1, \dots, H_n$.
To display these constants of motion we first define the matrices
\begin{equation}\label{eq:MatrixNotation}
  \begin{split}
    P &= \diag(m_1,\dots,m_n),\\
    E &= (E_{ij})_{i,j=1}^n, \qquad\text{where $E_{ij}=e^{-\abs{x_i-x_j}}$},\\
    T &= (T_{ij})_{i,j=1}^n, \qquad\text{where $T_{ij}=1+\sgn(i-j)$},\\
  \end{split}
\end{equation}
where $P$ and~$E$ depend on the variables appearing in~\eqref{eq:NovikovODE}.
One can then show that for every complex $\lambda$,
the polynomial $A(\lambda)=\det( I - \lambda TPEP)$ is a constant of motion for
equations~\eqref{eq:NovikovODE}, which implies that the coefficient
of $\lambda^{k}$ in the polynomial $A(\lambda)$ is a
constant of motion. It can also be shown that these coefficients
(for $1 \le k \le n$) are in fact Poisson commuting and functionally independent,
thereby providing the desired set of $n$ constants of motion.
By elementary linear algebra, the coefficient of $\lambda^k$
can be computed (up to a sign) as the sum over all $k\times k$
principal minors of the matrix~$TPEP$. However, before this result was found,
direct computations for small values of~$n$ had indicated that the
constants of motion ought to be given by the sums over all $k\times k$
minors of the symmetric matrix~$PEP$. It was the attempt to reconcile these
observations that led to the formulation of Theorem~\ref{thm:CanadaDay},
and as a result, the constants of motion now have the following description:

\begin{theorem}\label{thm:ConstantsOfMotion}
  The Novikov peakon ODEs \eqref{eq:NovikovODE} admit $n$
  constants of motion $H_1,\dots,H_n$,
  where $H_k$ equals the sum of all $k\times k$ minors (principal
  and non-principal) of the $n\times n$ symmetric matrix
  $PEP=(m_i m_j e^{-\abs{x_i-x_j}})_{i,j=1}^n$.
\end{theorem}

As a final remark, let us mention that explicit expressions for the
minors of $PEP$ can be written down easily with the help of the
Lindstr\"om--Gessel--Viennot lemma; see \cite{hls} for details.

\bibliographystyle{abbrv}
\bibliography{CDT}

\end{document}